\documentclass[12pt]{amsart}

\usepackage{amssymb,amsbsy}
\usepackage{latexsym,color,longtable}
\usepackage{verbatim,euscript}
\usepackage{fullpage}

\numberwithin{equation}{section}
\usepackage[colorlinks=true,linkcolor=blue,urlcolor=violet,citecolor=magenta]{hyperref}

\input {cyracc.def}

\newenvironment{proof*}
{\noindent {\sl Proof.}\quad }{\hfill    $\square$}

\catcode`\@=\active
\catcode`\@=11

\renewcommand{\@cite}[2]{[{{\bf #1}\if@tempswa , #2\fi}]}
\renewcommand{\@biblabel}[1]{[{\bf #1}]\hfill}
\catcode`\@=12

\newtheorem{thm}{Theorem}[section]
\newtheorem{lm}[thm]{Lemma}
\newtheorem{cl}[thm]{Corollary}
\newtheorem{prop}[thm]{Proposition}

\theoremstyle{remark}
\newtheorem{rmk}[thm]{Remark}

\theoremstyle{definition}

\newcommand {\be}{{\mathfrak b}}

\newcommand {\g}{{\mathfrak g}}

\newcommand {\mf}{{\mathfrak M}}
\newcommand {\n}{{\mathfrak n}}

\newcommand {\es}{{\mathfrak s}}
\newcommand {\te}{{\mathfrak t}}
\newcommand {\ut}{{\mathfrak u}}
\newcommand {\z}{{\mathfrak z}}

\newcommand {\slro}{{\mathfrak {sl}}_{r+1}}

\newcommand {\eus}{\EuScript}

\newcommand {\esi}{\varepsilon}
\newcommand {\ap}{\alpha}

\newcommand {\lb}{\lambda}
\newcommand {\vp}{\varphi}

\newcommand {\cP}{{\mathcal P}}
\newcommand {\N}{{\mathcal N}}
\newcommand {\co}{{\mathcal O}}

\newcommand {\ct}{{\eus Z}}

\newcommand {\BN}{{\mathbb N}}
\newcommand {\BQ}{{\mathbb Q}}
\newcommand {\BZ}{{\mathbb Z}}

\newcommand {\VV}{{\mathbb V}}

\newcommand {\sfr}{{\eus R}}

\newcommand {\cha}{{\mathsf{ch}}}

\newcommand {\eul}{{\mathsf{gec\,}}}

\newcommand {\hot}{{\mathsf{ht}}}

\newcommand {\Lie}{{\mathrm{Lie\,}}}

\newcommand {\tri}{{\mathfrak{sl}}_2}

\newcommand {\GR}[2]{{\textrm{{\color{blue}\bf #1}}}_{#2}}

\newcommand {\beq}{\begin{equation}}
\newcommand {\eeq}{\end{equation}}

\newcommand{\curge}{\succcurlyeq}
\newcommand{\curle}{\preccurlyeq}
\renewcommand{\le}{\leqslant}
\renewcommand{\ge}{\geqslant}

\newcommand {\bbk}{\Bbbk}

\newcommand{\vts}{\VV_{\theta_s}}
\newcommand{\vml}{\VV_\lb^\mu}
\newcommand{\mul}{m_\lb^\mu}
\newcommand{\mullq}{{\mf}_\lb^\mu(q)}
\newcommand{\multq}{{\mf}_\theta^\mu(q)}
\newcommand{\mulsq}{{\mf}_{\theta_s}^\mu(q)}
\newcommand{\PT}{{\mathfrak X}_+}
\newcommand{\XT}{\mathfrak X}

\begin{document}
\setlength{\parskip}{2pt plus 4pt minus 0pt}
\hfill {\scriptsize June 5, 2014} 
\vskip1ex

\title[On $q$-analogues of all weight multiplicities]{On Lusztig's $q$-analogues of all weight multiplicities of a representation}
\author[D.\,Panyushev]{Dmitri I. Panyushev}
\address{Institute for Information Transmission Problems of the R.A.S., Bol'shoi Karetnyi per. 19, 
127994 Moscow, Russia 
\hfil\break\indent
Independent University of Moscow, 
Bol'shoi Vlasevskii per. 11, 119002 Moscow, Russia
}
\email{panyushev@iitp.ru}
\thanks{Part of this work was done while I was visiting the Friedrich-Schiller-Universit\"at (Jena) in Spring 2013}
\maketitle

\section*{Introduction}
The ground field $\bbk$ is algebraically closed and of characteristic zero.
Let $G$ be a connected semisimple algebraic group, and  $T$ a maximal torus inside a Borel subgroup $B$. Write $\g,\te$, and $\be$ for their Lie algebras. If $\VV$ is a finite-dimensional rational
$G$-module, then $\VV=\oplus_{\mu\in\te^*} \VV^\mu$ is the weight decomposition with respect to
$T$ (or $\te$). If $\VV=\VV_\lb$ is a simple $G$-module with highest weight $\lb$, then $\mul=\dim(\vml)$. 
In this article, we present some results on Lusztig's $q$-analogues $\mullq$ of weight 
multiplicities $\mul$. 
The polynomial $\mullq$ is defined algebraically as an alternating sum over the Weyl group, through
the $q$-analogue of Kostant's partition function. Initially, Lusztig introduced $q$-analogues only for
dominant weights $\mu$~\cite[(9.4)]{Lus}. However, this constraint is unnecessary and $\mullq$ is a
non-trivial polynomial for any $\mu$ such that $\lb-\mu$ is a linear combination of positive roots with 
nonnegative coefficients; in particular, for all weights of $\VV_\lb$. 
A relationship with certain Kazhdan-Lusztig polynomials \cite{kato} implies that $\mullq$ has nonnegative coefficients whenever $\mu$ is dominant. For instance,
if $\VV_\lb$ has the zero weight, with $m_\lb^0=n$, then
$\mf_\lb^0(q)=\sum_{i=1}^n q^{m_i(\lb)}$ and $m_1(\lb),\dots,m_n(\lb)$ are the 
{\it generalised exponents\/} of $\VV_\lb$. These numbers were first considered by 
Kostant~\cite[n.\,5]{ko63} 
in connection with the graded $G$-module structure of the ring $\bbk[\N]$, where $\N\subset\g$ is the 
nilpotent cone. The interpretation of Kostant's generalised exponents via polynomials $\mf_\lb^0(q)$ 
is due to W.\,Hesselink \cite{he80} and  D.\,Peterson (unpublished).

In Section~\ref{sect:recoll-q-an}, we gather basic properties of polynomials $\mullq$ and recall their 
relationship to cohomology of line bundles on $G\times_B\ut$. We emphasise the 
role of results of Broer on the non-negativity of coefficients of $\mullq$ \cite{br93} and the induction 
lemma for computing $\mullq$ \cite{br94}. Using Broer's results allows us to quickly recover  some 
known results on coefficients of degenerate Cherednik kernel that appear in work of Bazlov, Ion, and 
Viswanath \cite{ba01,ion04,visw}. We also prove that $\mf_\lb^\mu(q+1)$ is a polynomial in $q$ with 
nonnegative coefficients.

In Section~\ref{sect:adj+little}, $\g$ is assumed to be simple, and then $\theta$ is the highest root. 
We determine Lusztig's $q$-analogues for all roots of $\g=\VV_\theta$. Furthermore,
if $\g$ has two root lengths, then the short dominant root
$\theta_s$ determines a representation that is called {\it little adjoint}, and we also compute 
$q$-analogues for all  weights of $\vts$. Then we obtain a formula for 
the weighted sum $\sum_{\mu} m_\theta^\mu\multq$, which implies that it depends only on 
$\mf_\theta^0(q)$ and the Coxeter number of $\g$. A similar result is valid  for $\vts$.

In Section~\ref{sect:summa}, we prove that, for any simple $G$-modules $\VV_\lb$ and 
$\VV_\gamma$, the sum $\sum_{\mu} m_\gamma^\mu\mullq$ is equal to the $q$-analogue of
the zero weight multiplicity for the (reducible) $G$-module $\VV_\lb\otimes\VV_\gamma^*$~(Theorem~\ref{thm:main}).
Therefore, $\sum_{\mu} m_\gamma^\mu\mullq=\sum_{\mu} m_\lb^\mu\mf_\gamma^\mu(q)$ and
this also provides another formula for the $\BZ[q]$-valued symmetric bilinear form on the character 
ring of $\g$ that was introduced by R.\,Gupta (Brylinski) in \cite{rkg87-A}. 
As a by-product, we obtain that such a weighted sum is always a polynomial with non-negative
coefficients. Comparing two formulae for $\sum_{\mu} m_{\theta_s}^\mu\mf_{\theta_s}^\mu(q)$
yields a curious identity involving the Poincar\'e polynomial for $W_{\theta_s}$, the Weyl group
stabiliser of $\theta_s$, and $\mf_{\theta_s}^0(q)$ (Corollary~\ref{cor:tozhdestvo-theta}).
We hope that there ought to be other interesting results pertaining to $q$-analogues 
of all weights of a representation.

If $\g$ is simple and $\eta_i$ is the number of positive roots of height $i$, then the partition formed by 
the exponents of $\g$ is dual (conjugate) to the partition formed by the $\eta_i$'s, 
see~\cite{ko59,ion04,visw}. 
Section~\ref{sect:jump-and-general} contains a geometric explanation and generalisation to this result.
Let $e\in\g$ be a principal nilpotent element. We prove that if $\dim\VV_\lb^e=\dim\VV_\lb^\te$, then
the `positive' weights of $\VV_\lb$ exhibit the similar phenomenon relative to the generalised 
exponents of $\VV_\lb$.

\subsection*{Main notation}
Throughout, $G$ is a connected  semisimple algebraic group with $\Lie G=\g$.
We fix a Borel subgroup $B$ and a maximal torus $T\subset B$, and consider the corresponding
triangular decomposition $\g=\ut\oplus\te\oplus\ut^-$, where $\Lie B=\ut\oplus\te$. Then

{\bf --} \ $\Delta$ is the {\it root system\/} of $(\g,\te)$, $\Delta^+$  is the set of positive roots 
corresponding to $\ut$,  $\Pi=\{\ap_1,\ldots,\ap_r\}$ is the set of simple roots in $\Delta^+$, 
and $\rho=\frac{1}{2}\sum_{\mu\in\Delta^+}\mu$;

{\bf --} \ $\XT$ is the lattice of integral weights of $T$ and 
${\te}^*_{\BQ}$ is the $\BQ$-vector subspace of $\te^*$ generated by $\XT$,
$Q=\oplus _{i=1}^r {\Bbb Z}\ap_i \subset \XT$ is the {\it root lattice}, and
$Q_+$ is the {monoid\/} generated by $\ap_1,\dots,\ap_r$.
If $\gamma=\sum_{i=1}^r c_i\ap_i\in Q_+$, then $\hot(\gamma)=\sum_{i=1}^r c_i$ is the {\it height\/} of $\gamma$.

{\bf --} \  $\PT$ is the monoid of dominant weights and $\vp_i\in\PT$ is the fundamental weight corresponding to $\ap_i\in\Pi$;

{\bf --} \ $W$ is the {\it Weyl group\/} of $(\g,\te)$ and $(\ ,\ )$ is a $W$-invariant positive-definite 
inner product on ${\te}^*_{\BQ}$. 
As usual, $\mu^\vee={2\mu}/{(\mu,\mu)}$ is the coroot for $\mu\in \Delta$.

{\bf --} \ If $\lb\in\PT$, then $\VV_\lb$ is the simple
$G$-module with highest weight $\lb$, $\VV_\lb^*$ is its dual, and $\lb^*\in\PT$ is defined by
$\VV_{\lb^*}=\VV_\lb^*$.

\noindent
For $\ap\in\Pi$, we let $s_\ap$ denote the corresponding simple reflection in $W$. 
If $\ap=\ap_i$, then we also write $s_i=s_{\ap_i}$. 
The {\it length function\/} on $ W$ with respect to  $s_1,\dots,s_r$ is  denoted by $\ell$. 

\section{Generalities on $q$-analogues of weight multiplicities} 
\label{sect:recoll-q-an}

\noindent
If $\lb\in\PT$, then $\vml$ is the $\mu$-weight space of $\VV_\lb$, $m_\lb^\mu=\dim\vml$, and
$\chi_\lb=\cha(\VV_\lb)=\sum_\mu m_\lb^\mu e^\mu \in \BZ[\mathfrak X]$ is the {\it character\/} of $\VV_\lb$. Let $\esi(w)=(-1)^{\ell(w)}$ be the sign of $w\in W$.
By Weyl's character formula, 
$\cha(\VV_\lb)=\displaystyle
\frac{\sum_{w\in W}\esi(w)e^{w(\lb+\rho)}}{e^\rho\prod_{\gamma\in\Delta^+}(1-e^{-\gamma}) }$. \ 
For $\mu,\gamma\in\XT$, we write $\mu\curle\gamma$, if 
$\gamma-\mu\in Q_+$. 

Define functions ${\mathcal P}_q(\mu)$ by the equation
\[ \frac{1}{\prod_{\ap\in \Delta_+}(1-qe^\ap )}=:\sum_{\mu\in Q_+}
{\mathcal P}_q(\mu)e^\mu\ . \]
Then $\mathcal P_q(\mu)$ is a polynomial in $q$ with $\deg\mathcal P_q(\mu)=\hot(\mu)$ and 
$\mu \mapsto {\mathcal P}(\mu):={\mathcal P}_q(\mu)\vert_{q=1}$ is the usual Kostant's partition
function. For $\lb,\mu\in\PT$,
Lusztig \cite[(9.4)]{Lus} (see also \cite[(1.2)]{kato}) introduced a fundamental $q$-analogue of weight multipliciities $\mul$:
\beq    \label{eq:def-q-an}
   \mullq=\sum_{w\in W}\esi(w){\mathcal P}_q(w(\lb+\rho)-(\mu+\rho)) . 
\eeq
For series $\GR{A}{r}$, these are the classical {\it Kostka-Foulkes polynomials}. Therefore, this name
is sometimes used in the general situation. 
It is also known that $\mullq$ are related to certain Kazhdan-Lusztig polynomials associated with the
corresponding affine Weyl group \cite{Lus}, \cite[Theorem\,1.8]{kato}.
However, one needn't restrict oneself with only dominant weights $\mu$, and  the polynomials 
$\mullq$ can be considered for arbitrary $\mu\in\XT$. 
It is easily seen that
\begin{itemize}
\item  $\mullq\equiv 0$ unless $\lb\curge\mu$;
\item if $\lb\curge\mu$, then $\mullq$ is a monic polynomial and $\deg\mullq=\hot(\lb-\mu)$; therefore,
$\mf_\lb^\lb(q)\equiv 1$;
\item  $\mf_\lb^\mu(1)=m_\lb^\mu$.
\end{itemize}
\noindent
In particular, if $\mu\curle\lb$, but $\mu$ is not a weight of $\VV_\lb$, then $\mf_\lb^\mu(1)=0$ and
therefore $\mullq$ has negative coefficients. If $\mu$ is dominant, then the relationship with 
Kazhdan-Lusztig polynomials implies that $\mullq$ has nonnegative coefficients. The most general 
result on non-negativity of the coefficients of $\mullq$, whose proof exploits the cohomological 
interpretation, is due to Broer~\cite{br93}, see Theorem~\ref{broer-krit} below.

\subsection{A relationship to cohomology of line bundles} \leavevmode\par
\noindent
Let $\ct$ be the cotangent bundle of $G/B$, i.e., $\ct=G\times_B\ut$. Recall that the corresponding 
collapsing $\ct\to G\ut=:\N\subset \g$ is birational and 
$H^0(\ct,\co_\ct)=\bbk[\N]$ \cite{he76}. Here $\N$ is the cone of nilpotent elements of $\g$.
For $\mu\in\XT$, let $\bbk_\mu$ denote the corresponding 
one-dimensional $B$-module. We consider line bundles on $\ct$ induced from homogeneous line 
bundles on $G/B$, i.e., line bundles of the form
\[
    G\times_B(\ut\oplus\bbk_\mu) \to G\times_B\ut=\ct .
\] 
The (invertible) sheaf of section of this bundle is denoted by $\mathcal L_\ct(\bbk_\mu)$. More generally, 
if $N$ is a rational $B$-module, then 
\[
    G\times_B(\ut\oplus N) \to G\times_B\ut=\ct 
\]
is a vector bundle on $\ct$ of rank $\dim N$ and the corresponding sheaf of sections (locally free 
$\co_\ct$-module) is  $\mathcal L_\ct(N)$. If $\mathcal E$ is a locally free $\co_\ct$-module,
then $\mathcal E^\star$ is its dual. For instance, 
$\mathcal L_\ct(N)^\star=\mathcal L_\ct(N^*)$, where $N^*$ is the dual $B$-module.

The cohomology groups of $\mathcal L_\ct(N)$  
have a natural structure of a graded $G$-module by
\[
    H^i(G\times_B \ut, \mathcal L_{G\times_B \ut}(N))\simeq \bigoplus_{j=0}^\infty
    H^i(G/B, \mathcal L_{G/B}(\mathcal S^j \ut^*\otimes N)) ,
\]
where $\mathcal S^j \ut^*$ is the $j$-th symmetric power of the dual of $\ut$. 
Set $H^i(\mu):=H^i(\ct, \mathcal L_{\ct}(\mu)^\star)$.
It is a graded $G$-module with
\[
   (H^i(\mu))_j=H^i(G/B, \mathcal L_{G/B}(\mathcal S^j \ut\otimes \bbk_\mu)^\star).
\]
As $\dim (H^i(\mu))_j < \infty$, the graded character 
of $H^i(\mu)$ is well-defined:
\[
     \cha_q( H^i(\mu))=\sum_j \sum_{\lb\in\PT} \dim \mathsf{Hom}_G\bigl(\VV_\lb, (H^i(\mu))_j\bigr)
     \chi_\lb q^j \in \BZ [\mathfrak X] [[q]] .
\]
The reader is referred to work of Broer and Brylinski for more details \cite{br93, br94, rkb89}.

\begin{thm}[{\cite[Lemma\,6.1]{rkb89}}]   \label{thm:ranee-svyaz}
For any $\mu\in\XT$, we have 
\[
   \displaystyle\sum_{i}  (-1)^i \cha_q ( H^i(\mu))=
   \sum_{\lb\in\PT}\mf_{\lb}^\mu(q) \chi_{\lb}^* .
\]
\end{thm}
\noindent A more general version of this relation, where $\n\subset\g$ is replaced with a $B$-stable subspace of an arbitrary $G$-module $\VV_\lb$, appears in \cite[Theorem\,3.8]{selecta}.

For $\mu=0$, we have $\mathcal L_{\ct}(0)=\co_\ct$ and $H^i(\ct,\co_{\ct})=0$ for $i>0$ \cite{he76}.
Therefore, the sum $\sum_{\lb\in\PT}\mf_{\lb}^0(q) \chi_{\lb}^*$
represents the graded character of $H^0(\ct,\co_{\ct})\simeq \bbk[\N]$~\cite{he80}.

For $\mu\in\XT$, we write $\mu^+$ for the unique 
element in $W\mu\cap\PT$.

\begin{thm}[Broer's criterion \cite{br93,br97}]    \label{broer-krit}
The following conditions are equivalent for $\mu\in\XT$:
\begin{itemize}
\item[\sf (1)] \ $\mullq$ has nonnegative coefficients for all $\lb\in\PT$; 
\item[\sf (2)] \ if $\mu\curle \gamma\curle \mu^+$ and $\gamma\in\PT$, then $\gamma=\mu^+$;
\item[\sf (3)] \ $(\mu,\nu^\vee)\ge -1$ for all $\nu\in\Delta^+$.
\end{itemize}
\end{thm}
\noindent
The equivalence of (1) and (2) is proved in \cite[Theorem\,2.4]{br93}; the underlying reason is that,
for such $\mu$, higher cohomology of $\mathcal L_{\ct}(\mu)^\star$ vanishes. 
The equivalence of (2) and (3) 
appears in \cite[Prop.\,2(iii)]{br97}.
\begin{rmk}
The required equivalence of (2) and (3) is correctly proved by Broer, but some other 
assertions of Proposition~2 in \cite{br97} are false. Namely, in part (iii) Broer claims the equivalence of certain conditions (a),(b), and (c), where (a) and (b) are just our conditions (2) and (3). But condition (c) must be excluded from that list. Moreover, part (ii) in \cite[Prop.\,2]{br97} is also false.
A common counterexample is given e.g. by $\beta=-\vp_1$ for $\g=\slro$, $r\ge 2$. This $\beta$ 
satisfies Broer's  conditions (a) and (b), but not (c); and part (ii) also fails for $\beta$. More generally,
if $\kappa\in\PT$ is minuscule, then $\beta=-\kappa$ provides a counterexample to Broer's 
assertions.
\end{rmk}
Recall that $m_\lb^0\ne 0$ if and only if $\lb\in\PT\cap Q$. Then 
$\mf_\lb^0(q)=\sum_{j=1}^n q^{m_j(\lb)}$ ($n=m_\lb^0$) is a polynomial with nonnegative coefficients 
and the integers $m_1(\lb),\dots,m_n(\lb)$ are called the {\it generalised exponents\/} of $\VV_\lb$.
If $\g$ is simple and $\VV_\lb=\g$, then they coincide with the usual exponents of $\g$ (= of $W$)~\cite{he80}.

\subsection{Broer's induction lemma and degenerate Cherednik kernel}

The following fundamental result of Broer is a powerful tool for computing $q$-analogues of weight 
multiplicities. Unfortunately, it did not attract  the attention it deserves. Perhaps the reason is that
Broer formulates it as a relation in ``the Grothendieck group of finitely generated graded $\bbk[\N]$-modules with a compatible $G$-module structure''.
However, extracting the coefficients of $\chi_\lb^*$, one obtains the following down-to-earth description:

\begin{thm}[Induction Lemma, cf.~{\cite[Prop.\,3.15]{br94}}]  \label{induction-lemma}
Let $\lb\in \PT$. If\/ $\gamma\in\XT$ and $(\gamma,\ap^\vee)=-n<0$ for some $\ap\in\Pi$
(hence $s_\ap(\gamma)=\gamma+n\ap$),  then
\beq   \label{eq:ind-lemma}
    \mf_\lb^{\gamma}(q)+\mf_\lb^{s_\ap(\gamma)-\ap}(q)=q(\mf_\lb^{\gamma+\ap}(q)+\mf_\lb^{s_\ap(\gamma)}(q)) .
\eeq
\end{thm}
\noindent 
In particular, for $n=1$, this formula contains only $\gamma$ and $\gamma+\ap$ and one merely 
obtains $\mf_\lb^\gamma(q)=q\, \mf_\lb^{\gamma+\ap}(q)$.
Broer's proof of the Induction Lemma exploits the cohomological interpretation of Lusztig's 
$q$-analogues discussed above, and includes the passage from $G/B$ to $G/P_\ap$, where
$P_\ap$ is the minimal parabolic subgroup corresponding to $\ap$.

Actually, the name "Induction Lemma" is assigned in~\cite{br94} to a certain preparatory result. But, we feel that it is more appropriate to associate such a name with Broer's Proposition~3.15.

It is observed in \cite[5.1]{rkg87} that Lusztig's $q$-analogues $\mf_{\lb}^\mu(q)$ satisfy the identity
\begin{equation}    \label{eq:series-Lqa}
   \sum_{\mu:\,\mu\curle \lb} \mf_{\lb}^\mu(q) e^\mu=
   \frac{\sum_{w\in W}\esi(w)e^{w(\lb+\rho)} }{e^\rho\prod_{\gamma\in\Delta^+}(1-qe^{-\gamma})}=
   \chi_\lb\cdot \prod_{\gamma\in\Delta^+} \frac{(1-e^{-\gamma})}{(1-qe^{-\gamma})}=\chi_\lb \xi_q \ .
\end{equation}
Here $\xi_q=\displaystyle\prod_{\gamma\in\Delta^+}\frac{1-e^{-\gamma}}{1-qe^{-\gamma}}$ is the
{\it degenerate Cherednik kernel\/}, and for $\lb=0$ one obtains 
\beq  \label{eq:cherednik-kern}
    \xi_q=\sum_{\mu\in Q_+}\mf_0^{-\mu}(q)e^{-\mu} .
\eeq
Thus, the coefficients of $\xi_q$ are certain Lusztig's $q$-analogues.
As an application of the Induction Lemma, 
we easily recover some known results on coefficients of $\xi_q$,
cf. Bazlov~\cite[Theorem\,3]{ba01}, 
Ion~\cite[Eq.\,(5.35)]{ion04}, and Viswanath~\cite[Prop.\,1]{visw}.

\begin{prop}  \label{prop:coef-cherednik}
If $\mu\in\Delta^+$, then $[e^{-\mu}](\xi_q)=\mf_0^{-\mu}(q)=q^{\hot(\mu)}- q^{\hot(\mu)-1}$.
\end{prop}
\begin{proof}
 We argue by induction on $\hot(\mu)$.
 
 1) {\sl Base}: if $\mu\in\Delta^+$ is simple, then it is easily seen that $[e^{-\mu}](\xi_q)=q-1$.
 
 2) {\sl Step}: Suppose that $\hot(\mu)\ge 2$ and the assertion holds for all $\gamma\in\Delta^+$ with
 $\hot(\gamma)< \hot(\mu)$. Take any $\ap\in\Pi$ such that $(\mu, \ap^\vee)=n>0$. Then
 $s_\ap(\mu)=\mu-n\ap\in\Delta^+$ and applying \eqref{eq:ind-lemma} with $\gamma=-\mu$ we obtain
 \[
    \mf_0^{-\mu}(q)+\mf_\lb^{-\mu+(n-1)\ap}(q)=q(\mf_\lb^{-\mu+\ap}(q)+\mf_\lb^{-\mu+n\ap}(q)) .  
 \]
Since $\mu-\ap$ and  $\mu-(n-1)\ap$ are also positive roots, of smaller height,  using the induction 
assumption yields the desired expression for $\mf_0^{-\mu}(q)$.
\end{proof}

\begin{rmk}   \label{rmk:B-Ion-V}
Bazlov and Ion work with the usual (2-parameter) Cherednik kernel, and then specialise their formulae
to one-parameter case. They  use the general theory of Macdonald polynomials, whereas 
Viswanath provides a direct elementary approach to computing coefficients of  $\xi_q$. One
can notice that  Viswanath's note \cite{visw} contains implicitly an inductive formula for the coefficients of 
$\xi_q$. His argument basically proves that if $\beta\in Q_+$ and $s_i(\beta)=\beta-k\ap_i$ ($k>0$), then
\beq   \label{eq:tozhd-V}
   \mf_0^{-\beta}(q)=(q-1)\sum_{j=1}^{k-1}\mf_0^{-\beta+j\ap_i}(q)+ q{\cdot}\mf_0^{-s_i(\beta)}(q) .
\eeq
Actually, one needn't assume here that $s_i(\beta)\in Q_+$. If some of $\beta-j\ap_i$ do not belong
to $Q_+$, then the corresponding $q$-analogues are replaced by zero.
It is a simple exercise to deduce \eqref{eq:tozhd-V} from \eqref{eq:ind-lemma} with $\lb=0$, and vice versa.
[Left to the reader.]
\end{rmk}

Substituting \eqref{eq:cherednik-kern} in the equality 
$\sum_{\mu:\,\mu\curle \lb} \mf_{\lb}^\mu(q) e^\mu=\chi_\lb \xi_q$,
we obtain 
\beq    \label{eq:mul-lb-and-0}
    \mullq=\sum_{\gamma:\, \gamma\curge \mu}  m_\lb^\gamma \mf_0^{\mu-\gamma}(q) ,
\eeq
so that all $q$-analogues for $\VV_\lb$ can (theoretically) be computed once we know enough 
coefficients of $\xi_q$ and the usual weight multiplicities. 
But even for the adjoint representation, this approach requires more than merely the knowledge of
$\mf_0^{-\nu}(q)$ for $\nu\in\Delta^+$. For, $\gamma-\mu$ need not be a root in the above formula.
However, Eq.~\eqref{eq:mul-lb-and-0} has a curious consequence.

\begin{lm}    \label{lem:q=0}
The polynomials $\mf_\lb^\mu(q+1)$ have nonnegative coefficients for all $\mu$.
If $\mu$ is a weight of\/ $\VV_\lb$ and $\mu\ne\lb$, then $\mf_\lb^\mu(0)=0$. 
\end{lm}
\begin{proof}
1) By the very definition of $\xi_q$, we have  
$
  \displaystyle \xi_{q}=\prod_{\gamma\in\Delta^+}\bigl(1+\sum_{n\ge 0}q^{n}(q-1)e^{(n+1)\gamma}\bigr)
$.
Whence all polynomials $\mf_0^{\nu}(q+1)$, the coefficients of $\xi_{q+1}$, 
have nonnegative coefficients. Using Eq.~\eqref{eq:mul-lb-and-0},  we carry it over to
arbitrary $\lb\in\PT$.

2) By Weyl's denominator formula, $\xi_q\vert_{q=0}=\sum_{w\in W} \esi(w)e^{w\rho-\rho}$. Therefore,
$\mf_0^{\nu}(0)=\esi(w)$ if $\nu=w\rho-\rho$, and is zero otherwise. Hence
$\mf_\lb^\mu(0)=\sum_{w\in W}\esi(w) m_\lb^{\mu+\rho-w\rho}$. For a weight $\mu$ of $\VV_\lb$, 
the latter equals $\delta_{\lb\mu}$ by Klimyk's formula, see e.g. \cite[\S\,3.8, Prop.\,C]{samel}.
\\ \, [One can also refer directly to Eq.~\eqref{eq:def-q-an}.]
\end{proof}

\section{All $q$-analogues  for the adjoint and little adjoint representations} 
\label{sect:adj+little}

\noindent
In this section, $\g$ is simple, $\theta$ is the highest root, and $\theta_s$ is the short dominant root 
in $\Delta^+$. Here we compute $q$-analogues for all weight multiplicities of the adjoint and little adjoint 
representations of $\g$ and show that their sum depends essentially only on the $q$-analogue of the zero
weight multiplicity and the Coxeter number of $\g$.

Let $m_i=m_i(\theta)$, $i=1,\dots,r$, be the {\it exponents\/} of (the adjoint representation of) $\g$
and $h$ the {\it Coxeter number\/} of $\g$.
We assume that $m_1\le m_2\le \ldots \le m_r$, hence $m_1=1$ and $m_r=h-1=\hot(\theta)$.
In the simply-laced case, all roots are assumed to be short. That is, the argument referring to long 
roots has to be omitted if $\g$ is of type {\sf A-D-E}.

\begin{thm}    \label{thm:all-q-adj}
 For any $\mu\in\Delta\cup\{0\}$, the polynomial $\multq$ depends only on $\hot(\theta-\mu)$, i.e., on
 $\hot(\mu)$. More precisely, 
 \begin{itemize}
\item[\sf (i)] \ $\mf_\theta^0(q)=q^{m_1}+\ldots+ q^{m_r}$;
\item[\sf (ii)] \ If $\mu\in\Delta^+$, then $\multq=q^{\hot(\theta-\mu)}=q^{h-1-\hot(\mu)}$;
\item[\sf (iii)] \ if $\ap\in\Pi$, then $\mf_\theta^{-\ap}(q)=(q-1)\mf_\theta^0(q)+q^{h-1}$;
\item[\sf (iv)] \   If $\mu\in\Delta^+$, then 
$\mf_\theta^{-\mu}(q)=q^{\hot(\mu)-1}{\cdot}\mf_\theta^{-\ap}(q)$. 
\end{itemize}
\end{thm}
\begin{proof}
(i)  This is well-known and goes back to  Hesselink~\cite{he80} and Peterson.
See also \cite[Theorem\,5.5]{ion04} and \cite[p.\,2]{visw}.

(ii) If $\mu\in\Delta^+$ is short, then $(\mu,\gamma^\vee)\ge -1$ for all $\gamma\in\Delta^+$ and therefore
$\multq$ has nonnegative coefficients by Broer's criterion
(Theorem~\ref{broer-krit}). 
Since $\deg\multq=\hot(\theta-\mu)$ and $m_\theta^\mu=1$, one has the only possibility for $\multq$.

If $\Delta$ has two root lengths and $\mu\in\Delta^+$ is long, then we argue by induction in 
$\hot(\theta-\mu)$. For $\mu=\theta$, one has
$\mf_\theta^\theta(q)=1$. To perform the induction step, assume that $\multq=q^{\hot(\theta-\mu)}$ for
some $\mu$ and $\mu\not\in\Pi$.
Then there is $\ap\in\Pi$ such that $(\ap,\mu)>0$ and hence $s_\ap(\mu)\in \Delta^+$ and 
$\hot(s_\ap(\mu))< \hot(\mu)$.
Here $\mu=s_\ap(\mu)+n\ap$ with $n\in\{1,2,3\}$, and by the Induction Lemma (Theorem~\ref{induction-lemma}) applied to $\gamma=s_\ap(\mu)$ we have
\[
    \mf_\theta^{s_\ap(\mu)}(q)+\mf_\theta^{\mu-\ap}(q)=q( \mf_\theta^{s_\ap(\mu)+\ap}(q)+
     \mf_\theta^{\mu}(q)) .
\]
For $n=1$, we immediately obtain that $\mf_\theta^{s_\ap(\mu)}(q)=q\multq=
q^{\hot(\theta-s_\ap(\mu))}$. 
For $n=2$ or $3$, we get the same conclusion using 
the fact that the roots $s_\ap(\mu)+\ap$ and $\mu-\ap$  are short  (and hence the corresponding 
$q$-analogues are already known).

(iii) Passing from $\ap\in\Pi$ to $-\ap$ (crossing over $0$) is also accomplished via the use of the 
Induction Lemma. Since $s_\ap(-\ap)=-\ap + 2\ap$, we have
\[
  \mf_\theta^{-\ap}(q)+\mf_\theta^0(q)=q(\mf_\theta^0(q)+\mf_\theta^\ap(q)) ,
\]
and it is already proved in part (ii) that $\mf_\theta^\ap(q)=q^{\hot(\theta-\ap)}=q^{h-2}$.

(iv) Going down from $-\ap$ ($\ap\in\Pi$), we again use the Induction Lemma. First, we prove the assertion for all negative short roots using the fact that  if $\mu\in \Delta^+$ is short and 
$\mu\ne\theta_s$, then there is $\ap\in\Pi$ such that $(\mu,\ap^\vee)=-1$ and hence
$s_\ap(-\mu)=-\mu-\ap$. Afterwards, we prove the assertion for the long roots,
as it was done in part (ii).
\end{proof}

\begin{rmk}
The simplest formula for $\mf_\theta^{-\ap}(q)$, $\ap\in\Pi$, occurs if $\g=\mathfrak{sl}_{r+1}$, where 
there are only three summands. Namely, $\mf_\theta^{-\ap}(q)=q^{r+1}+q^r-q$. But for 
$\g=\mathfrak{sp}_{2r}$ or $\mathfrak{so}_{2r+1}$, $r\ge 2$, we obtain 
$\mf_\theta^{-\ap}(q)=q^2+q^4+\ldots +q^{2r}-(q+q^3+\ldots+q^{2r-3})$.
\end{rmk}

The notation $\mu\dashv\VV_\lb$ means that $\mu$ is a weight of $\VV_\lb$.
\begin{thm}   \label{thm:sum-all-adj}
We have $\displaystyle\sum_{\mu\dashv \g} \multq=\mf_\theta^0(q)(\mf_\theta^0(q)-r+1)+ \frac{\mf_\theta^{0}(q)}{q}{\cdot} \frac{1-q^h}{1-q} $ or, equivalently,
\beq   \label{eq:sum-all-adj}
     \sum_{\mu\dashv \g}m_\theta^\mu \multq=\mf_\theta^0(q)^2+ \frac{\mf_\theta^{0}(q)}{q}\cdot
     \frac{1-q^h}{1-q} .
\eeq
\end{thm}
\begin{proof}
Since $m_\theta^\mu=1$ for $\mu\in\Delta$ and $m_\theta^0=r$, both formulae are equivalent. 
In fact, we compute separately the sums 
\[
   \eus S_+=\sum_{\mu\in\Delta^+} \multq \ \ \text{ and }\ \eus S_-=\sum_{\mu\in\Delta^-} \multq .
\]
Recall that the partition $(m_r,\dots,m_1)$ is dual to the partition  $(n_1,n_2,\dots)$, where 
$n_i=\#\{\gamma\in\Delta^+\mid \hot(\gamma)=i\}$ \cite{ko59}. Therefore, $\Delta^+$ can be partitioned into 
the strings of roots of lengths $m_1,m_2,\dots,m_r$ such that the $i$-th string contains the roots 
of height $1,2,\dots,m_i$. Then, by Theorem~\ref{thm:all-q-adj}(ii), the sum over the $i$-th string equals
\[
  q^{m_r-1}+q^{m_r-2}+\ldots + q^{m_r-m_i}=\frac{q^{m_r-m_i}-q^{m_r}}{1-q} .
\]
Since $m_i+m_{r-i+1}=m_r+1=h$, 
the total sum over $\Delta^+$ can be written as 
\[
   \eus S_+=\sum_{i=1}^r\frac{q^{m_i-1}-q^{h-1}}{1-q}=\sum_{i=1}^r\frac{q^{m_i}-q^{h}}{q(1-q)}=
   \frac{\mf_\theta^0(q)-r q^h}{q(1-q)} .
\]
Likewise, using the corresponding strings of negative roots, one proves that 
\[
  \eus S_-=\displaystyle((q-1)\mf_\theta^0(q)+q^{h-1}){\cdot}\frac{r-\mf_\theta^0(q)}{1-q} .
\]
It then remains to simplify the sums $\eus S_++\mf_\theta^0(q)+\eus S_-$ and
$\eus S_++r\mf_\theta^0(q)+\eus S_-$.
\end{proof}

Similar results are valid for the little adjoint representation of $G$. Let $\Delta_s$ denote the set of all short roots, hence $\{\theta_s\}=\Delta_s\cap\PT$. Set $\Pi_s=\Pi\cap\Delta_s$ and $l=\#(\Pi_s)$.
Recall that the set of weights of $\vts$ is $\Delta_s\cup\{0\}$, $m_{\theta_s}^0=l$, and $m_{\theta_s}^\mu=1$  for $\mu\in \Delta_s$.

The following observation is a particular case of \cite[Theorem\,5.5]{ion04}, and we provide a proof
for reader's convenience.
\begin{lm}    \label{lm:zero-q-little}
Let $(n_{1s},n_{2s},\dots)$ be the partition of $\#(\Delta_s^+)$ with $n_{i,s}=\#\{\gamma\in\Delta_s^+\mid
\hot(\gamma)=i\}$, in particular, $\eta_{1,s}=l$.
If $(e_1,e_2,\dots,e_l)$ is the dual partition,
then $\mf_{\theta_s}^0(q)=q^{e_1}+\ldots+ q^{e_l}$.
\end{lm}
\begin{proof}
By Proposition~\ref{prop:coef-cherednik} and \eqref{eq:mul-lb-and-0}, we have
\[
\mf_{\theta_s}^0(q)=\sum_{\gamma\in Q_+} m_{\theta_s}^\gamma \mf_0^{-\gamma}(q)=l+\sum_{\mu\in\Delta_s^+}(q^{\hot(\mu)}-q^{\hot(\mu)-1}) .
\]
Since $\#(\Pi_s)=l$, the term $l$ cancels out and the coefficient of $q^j$ equals 
$n_{j,s}-n_{j+1,s}$  for $j\ge 1$. On the other hand, the number of parts $j$ in the dual partition also equals
$n_{j,s}-n_{j+1,s}$.
\end{proof}

An easy verification shows that, for the root systems with two root lengths, the generalised exponents 
$e_1,\dots,e_l$ of the little adjoint representation are:

$\GR{B}{n}$ -- \ $n\ (l=1)$;  $\GR{C}{n}$ -- \ $2,4,\dots,2n{-}2\ (l=n{-}1)$; $\GR{F}{4}$ -- \ $4,8\ (l=2)$; 
$\GR{G}{2}$ -- \ $3\ (l=1)$.

\noindent In particular, if $e_1\le\dots \le e_l$, then $e_i+e_{l+1-i}=h$ for all $i$.
\begin{thm}    \label{thm:all-q-little}
For any $\mu\in\Delta_s\cup\{0\}$, the polynomial $\mulsq$ depends only on $\hot(\theta_s-\mu)$, i.e., on
 $\hot(\mu)$. More precisely, 
 \begin{itemize}
\item[\sf (i)] \ $\mf_{\theta_s}^0(q)=q^{e_1}+\ldots+ q^{e_l}$;
\item[\sf (ii)] \ If $\mu\in\Delta_s^+$, then $\mulsq=q^{\hot(\theta_s-\mu)}$;
\item[\sf (iii)] \ if $\ap\in\Pi_s$, then $\mf_{\theta_s}^{-\ap}(q)=(q-1)\mf_{\theta_s}^0(q)+q^{\hot(\theta_s)}$;
\item[\sf (iv)] \   If $\gamma\in\Delta_s^+$, then 
$\mf_\theta^{-\gamma}(q)=q^{\hot(\gamma)-1}{\cdot}\mf_{\theta_s}^{-\ap}(q)$. 
\end{itemize}

\end{thm}
\begin{proof}
Part (i) is the subject of Lemma~\ref{lm:zero-q-little}. The proof of other parts is similar to those of 
Theorem~\ref{thm:all-q-adj}. 
\end{proof}

\begin{thm}   \label{thm:sum-all-little}
We have $\displaystyle\sum_{\mu\dashv \vts} \mulsq=
\mf_{\theta_s}^0(q)(\mf_{\theta_s}^0(q)-l+1)+ 
\frac{\mf_\theta^{0}(q)}{q^{h-\hot(\theta_s)}}{\cdot} \frac{1-q^h}{1-q} $ or, equivalently,
\beq   \label{eq:sum-all-little}
     \sum_{\mu\dashv \vts}m_{\theta_s}^\mu \mulsq=\mf_{\theta_s}^0(q)^2+ 
     \frac{\mf_{\theta_s}^{0}(q)}{q^{h-\hot(\theta_s)}}\cdot \frac{1-q^h}{1-q} .
\eeq
\end{thm}
\begin{proof}
Our argument is similar to that of Theorem~\ref{thm:sum-all-adj}.  
Since $(e_l,e_{l-1},\dots,e_1)$ and $(n_{1,s},n_{2,s},\dots)$ are dual partitions,
we present $\Delta^+_s$ as a union of $l$ 
strings of roots, where the $i$-th string consists of roots of height $1,2,\dots,e_i$.
Then, using the fact that $e_j+e_{l-j+1}=h$ for all $j$,
one computes that the sums of $q$-analogues of weight multiplicities over $\Delta_s^+$ and $-\Delta_s^+$ are equal to 
\ $\displaystyle \frac{\mf_{\theta_s}^0(q)-l q^h}{q^{h-\hot(\theta_s)}(1-q)}$ \ 
and \ $\displaystyle ((q-1)\mf_{\theta_s}^0(q)+q^{\hot(\theta_s)})\frac{l-\mf_{\theta_s}^0(q)}{(1-q)}$, \  respectively.
\end{proof}

\begin{rmk}   \label{rem:polinom}
It follows from this theorem that $\frac{\mf_{\theta_s}^{0}(q)}{q^{h-\hot(\theta_s)}}$ is a polynomial in $q$.
\end{rmk}

\begin{rmk}   \label{rem:sovpad}
In the simply-laced case, we have $l=r$, $\theta=\theta_s$, and $h-\hot(\theta_s)=1$. Then Theorems~\ref{thm:sum-all-adj} and \ref{thm:sum-all-little} yield the same formulae.
\end{rmk}

\begin{rmk}   \label{rem:subreg-orb}
Recall that the singular locus $\N^{sg}$ of $\N$ is irreducible (and the dense $G$-orbit in
$\N^{sg}$ is said to be subregular.) For any $\ap\in\Pi_s$ and $\lb\in Q\cap\PT$, Broer proves that
the collapsing $G\times_{P_\ap}\n_\ap\to \N^{sg}$, where $\n_\ap$ is 
the nilradical of $\Lie P_\ap$, is birational and
$\mf_\lb^0(q)- q^{\hot(\ap^+)}\mf_\lb^{\ap^+}(q)$ is the Poincar\'e polynomial counting the occurrences
of $\VV_\lb^*$ in the graded ring $\bbk[\N^{sg}]$, i.e.,
\[
   \mf_\lb^0(q)- q^{\hot(\ap^+)}\mf_\lb^{\ap^+}(q)=\sum_i \dim_\bbk \mathsf{Hom}_G(\VV_\lb^*,
   \bbk[\N^{sg}]_i)\, q^i 
\]
\cite[Cor.\,4.7]{br93}. In particular, $m_\lb^0-m_\lb^\ap$ is the multiplicity of $\VV_\lb^*$ in 
$\bbk[\N^{sg}]$. Using the Induction Lemma, we can prove that
$q\,\mf_\lb^\ap(q)=q^{\hot(\ap^+)}\mf_\lb^{\ap^+}(q)$. Therefore, this Poincar\'e polynomial is also equal 
to $\mf_\lb^0(q)-q\,\mf_\lb^\ap(q)=q\,\mf_\lb^0(q)-\mf_\lb^{-\ap}(q)$.
\\ \indent
For the long simple root $\ap$, the collapsing $G\times_{P_\ap}\n_\ap\to \N^{sg}$
is not birational and the ring $\bbk[\N^{sg}]$ should be replaced with 
$\bbk[G\times_{P_\ap}\n_\ap]$.
\end{rmk}

\section{A weighted sum of $q$-analogues of all weight multiplicities}
\label{sect:summa}

\noindent
For a (possibly reducible) $G$-module $V=\sum_j a_j\VV_{\lb_j}$, we set
$\mf_{V}^0(q)=\sum_j a_j \mf_{\lb_j}^0(q)$. In~\cite{rkg87-A}, R.~Gupta~(Brylinski) considered a
$\BZ[q]$-valued symmetric bilinear form $ \langle\!\langle\ ,\ \rangle\!\rangle$
on the character ring of $\g$:
\[
    \langle\!\langle \cha(V_1),\cha(V_2)\rangle\!\rangle=\mf_{V_1\otimes V_2^*}^0(q) .
\]
She proved that this form has a nice expression via the $q$-analogues of {\sl dominant\/} weights 
occurring in both $V_1$ and $V_2$. For any $\nu\in\PT$, consider the stabiliser $W_\nu\subset W$ 
and the restriction of the length function $\ell$ to $W_\nu$. Set 
$t_\nu(q)=\sum_{w\in W_\nu} q^{\ell(w)}$, the Poincar\'e polynomial of $W_\nu$. In particular,
$t_0(q)$ is the Poincar\'e polynomial of $W$.

\begin{thm}[{\cite[Cor.\,2.4]{rkg87-A}}]  \label{thm:rkb-a}
For all $\lb,\gamma\in \PT$, one has
\[
  \langle\!\langle \cha(\VV_\lb),\cha(\VV_{\gamma})\rangle\!\rangle=\sum_{\nu\in\PT}
  \mf_{\lb}^\nu(q)\mf_{\gamma}^\nu(q)\frac{t_0(q)}{t_\nu(q)} .
\]
\end{thm}

\noindent
We provide here another formula for this bilinear form that involves the usual weight 
multiplicities for one representation and $q$-analogues of {\sl all\/} weight multiplicities for the other representation.
Below, we write $\mf_{\lb^*\otimes\gamma}^0(q)$ in place of
$\mf_{\VV_{\lb^*}\otimes\VV_{\gamma}}^0(q)$.

\begin{thm}  \label{thm:main} 
For all $\lb,\gamma\in\PT$, we have
\[
    \sum_{\mu\dashv\VV_\gamma} m_\gamma^\mu\mullq=\sum_{\mu\dashv\VV_\lb} m_\lb^\mu\mf_\gamma^\mu(q)=\mf_{\lb^*\otimes\gamma}^0(q)=
    \sum_{\nu\in\PT}\mf_\gamma^\nu(q)\mf_\lb^\nu(q){\cdot}\frac{t_0(q)}{t_\nu(q)} .
\]
\end{thm}
\begin{proof}
The last equality here is the above-mentioned result of R.~Brylinski; the first equality stems from the symmetry of the last expression with respect to $\lb$ and $\gamma$. Hence our task is to prove the
second equality.
Consider the vector bundle $G\times_B(\ut\oplus \VV_\lb^*)\to \ct$ and the corresponding sheaf 
$\mathcal L_\ct(\VV_\lb^*)$ of  graded $\bbk[\N]$-modules.
As in case of line bundles on $\ct$ (see Section~\ref{sect:recoll-q-an}), the graded character of $H^i(\ct,\mathcal L_\ct(\VV_\lb)^\star)$ is well-defined and 
we say that 
\[
    \eul(\VV_\lb)=\sum_i (-1)^i\cha_q\bigl(H^i(\ct,\mathcal L_\ct(\VV_\lb)^\star)\bigr)
\]
is the  {\it graded Euler  characteristic} (of $\mathcal L_\ct(\VV_\lb)^\star$). Let us compute 
$\eul(\VV_\lb)$ in two different ways.

First, we can replace $\VV_\lb$ with the completely reducible $B$-module
$\widetilde\VV_\lb=\oplus_\mu m_\lb^\mu \,\bbk_\mu$, which does not change the graded Euler 
characteristic. Then 
\beq    \label{eq:eul-1}
  \eul(\VV_\lb)=\eul(\widetilde\VV_\lb)=\sum_{\mu\dashv \VV_\lb}m_\lb^\mu\,\eul(\bbk_\mu)=
  \sum_{\mu\dashv \VV_\lb}\sum_{\nu\in\PT}m_\lb^\mu \mf_\nu^\mu(q)\chi_\nu^* ,
\eeq
where the last equality follows by Theorem~\ref{thm:ranee-svyaz}.
 
On the other hand, $\VV_\lb$ is a $G$-module, therefore $G\times_B(\ut\oplus \VV_\lb^*)\simeq
\ct\times \VV_\lb^*$ and
\beq    \label{eq:eul-2}
  \eul(\VV_\lb)\simeq  \cha(\VV_\lb^*) \cdot \eul(\bbk_0)=
  \chi_\lb^*{\cdot}\sum_{\nu\in\PT} \mf_\nu^0(q)\chi_\nu^* .
\eeq
Now, equating the coefficients of $\chi_\gamma^*$ in \eqref{eq:eul-1} and \eqref{eq:eul-2}, we will 
obtain the assertion. The required coefficient in \eqref{eq:eul-1} equals 
$\sum_{\mu\dashv \VV_\lb}m_\lb^\mu \mf_\gamma^\mu(q)$. Expanding the product
$\chi_\lb^*\chi_\nu^*=\sum_{\kappa\in\PT} c_{\lb^*\nu^*}^{\kappa}
\chi_\kappa$, we see that the coefficient of 
$\chi_\gamma^*$ in \eqref{eq:eul-2} equals $\sum_{\nu\in\PT} c_{\lb^*\nu^*}^{\gamma^*}\mf_\nu^0(q)$.
Since $c_{\lb^*\nu^*}^{\gamma^*}=c_{\lb^*\gamma}^\nu$, this sum also equals
$\mf_{\lb^*\otimes\gamma}^0(q)$. 
\end{proof}

\begin{cl}    \label{cor:lambda=gamma}
For any $\lb\in\PT$, we have
$\displaystyle
  \sum_{\mu\dashv\VV_\lb} m_\lb^\mu\mf_\lb^\mu(q)=\sum_{\nu\in\PT}\mf_\lb^\nu(q)^2{\cdot}\frac{t_0(q)}{t_\nu(q)}$.
\end{cl}

\noindent (Note that for $\nu\in\PT$, $\mf_\lb^\nu(q)$ is nonzero if and only if $\nu\dashv\VV_\lb$.)
This equality shows that the weighted sum 
$\sum_{\mu\dashv\VV_\lb} m_\lb^\mu\mf_\lb^\mu(q)$ is a more natural object  than just 
$\sum_{\mu\dashv\VV_\lb} \mf_\lb^\mu(q)$. Actually, we do not know any closed expression for the 
latter. Moreover, the weighted sum of $q$-analogues of all weight multiplicities is a polynomial with 
nonnegative coefficients, whereas this is not always the case for the plain sum
(use Theorem~\ref{thm:sum-all-adj} and look at the adjoint representation of $\mathfrak{sl}_{r+1}$ with $r\ge 4$).

\begin{rmk}   \label{rem:tempting}   
It was tempting to conjecture that Corollary~\ref{cor:lambda=gamma} could be refined so that one 
takes the sum over a sole Weyl group orbit $W\mu$ in the LHS and pick the summand corresponding 
to $\mu^+$ in the RHS. But this doesn't work! For instance, if $\mu=0\dashv\VV_\lb$, then the 
corresponding summands are $m_\lb^0\mf_\lb^0(q)$ (left) and $\mf_\lb^0(q)^2$ (right).
\end{rmk}

\begin{cl}   \label{cor:minuscule}
1) If\/ $\lb\in\PT$ is minuscule, then $\displaystyle\sum_{\mu\dashv\VV_\lb} \mf_\lb^\mu(q)=
\sum_{\mu\dashv\VV_\lb} q^{\hot(\lb-\mu)}=\frac{t_0(q)}{t_\lb(q)}$; \\
2) More generally, if\/ $\VV_\lb$ is weight multiplicity free (i.e., $m_\lb^\mu=1$ for all $\mu\dashv \VV_\lb$),
then \\ 
\centerline{$\displaystyle\sum_{\mu\dashv\VV_\lb} \mf_\lb^\mu(q)=
\sum_{\mu\dashv\VV_\lb} q^{\hot(\lb-\mu)}$.}
\end{cl}
\begin{proof}
1) In this case all weight multiplicities are equal to one  and $\lb$ is the only dominant weight of $\VV_\lb$.
Moreover, all the weights $\mu$ satisfy the condition (3) of Theorem~\ref{broer-krit} and therefore
$\mf_\lb^\mu(q)=q^{\hot(\lb-\mu)}$.

2) Since $m_\lb^\mu=1$ for all $\mu$, 
using Theorem~\ref{induction-lemma},  one easily proves by induction on $\hot(\lb-\mu)$ that $\mf_\lb^\mu(q)=q^{\hot(\lb-\mu)}$.
\end{proof}

This corollary shows that, for the weight multiplicity free case,  $
\sum_{\mu\dashv\VV_\lb} \mf_\lb^\mu(q)$
equals the {\it Dynkin polynomial} of $\VV_\lb$, see \cite[Sect.\,3]{endo04}.

\begin{cl}    \label{cor:tozhdestvo-theta}
(i) If\/ $\g$ is simply-laced, then \ 
$\displaystyle \frac{t_0(q)}{t_\theta(q)}=\frac{\mf_\theta^0(q)}{q}{\cdot}\frac{1-q^h}{1-q}$.
\\
(ii) \  
More generally, for any simple Lie algebra $\g$, we have
\ $\displaystyle \frac{t_0(q)}{t_{\theta_s}(q)}=\frac{\mf_{\theta_s}^0(q)}{q^{h-\hot(\theta_s)}}{\cdot}\frac{1-q^h}{1-q}$ \ and \\
$\displaystyle \frac{t_0(q)}{t_\theta(q)}=
\frac{\mf_{\theta^\vee}^0(q)}{q^{h-\hot(\theta^\vee)}}{\cdot}\frac{1-q^h}{1-q}$, where $\theta^\vee$ is 
regarded as the short dominant root in  $\Delta^\vee$.
\end{cl}
\begin{proof} \ 
In the simply laced case, $\theta=\theta_s$ and $\hot(\theta_s)=\hot(\theta^\vee)=h-1$. Therefore, it suffices to prove part (ii). For the first equality in (ii), we
combine Eq.~\eqref{eq:sum-all-little} and Corollary~\ref{cor:lambda=gamma} with $\lb=\theta_s$, and also use the fact that the only dominant weights of $\vts$ are $\theta_s$ and $0$. The second equality stems from the similar argument for the dual Lie algebra $\g^\vee$ and the fact that 
$t_\theta(q)=t_{\theta^\vee}(q)$.
\end{proof}

\begin{rmk}   \label{rem:obsuzhd}
The last corollary can be verified by a direct calculation.
Recall that if $d_1(\nu),\dots, d_r(\nu)$ are the degrees of basic invariants of the reflection group $W_\nu\subset GL(\te)$, then
$t_\nu(q)=\displaystyle\prod_{i=1}^r \frac{1-q^{d_i(\nu)}}{1-q}$. In particular, $d_i(0)=m_i+1$. 
It is a kind of miracle that 
the complicated fraction \ 
$\displaystyle \frac{t_0(q)}{t_{\theta_s}(q)}=\prod_{i=1}^r \frac{1-q^{d_i(0)}}{1-q^{d_i(\theta_s)}}$ 
simplifies to rather a simple expression!
\end{rmk}

\section{Generalised exponents and the height of weights}    
\label{sect:jump-and-general}

\noindent In Section~\ref{sect:adj+little}, we used the fact that the generalised exponents of 
$\g=\VV_\theta$ and $\vts$ are determined via the height of `positive weights (roots)'. Here we provide 
a geometric condition for this phenomenon and point out some other irreducible representation having 
the similar property. This relies on results of R.\,Brylinski on the principal filtration of a weight space
and `jump' polynomials~\cite{rkb89}.

Let $e$ be a principal nilpotent element of $\g$ and $\{e,\tilde h,f\}$ a corresponding principal 
$\tri$-triple in $\g$. Without loss of generality, we assume that $e$ is the sum of root vectors 
corresponding to  $\Pi$ and hence
$\ap(\tilde h)=2$ for all $\ap\in\Pi$ \cite{ko59,ko63}.
This means that upon the identification of $\te$ and $\te^*$, $\frac{1}{2}\tilde h$ is nothing but 
$\rho^\vee:=\frac{1}{2}\sum_{\gamma\in\Delta^+}\gamma^\vee$ and  
$\gamma(\frac{1}{2}\tilde h)=(\gamma,\rho^\vee)=\hot(\gamma)$ for all $\gamma\in Q_+$.

Let $\es=\langle e,\tilde h,f\rangle$ be the corresponding simple subalgebra of $\g$.
We write $\sfr_n$ for the simple $\es$-module of dimension $n+1$, so that the $\tilde h$-eigenvalues
in $\sfr_n$ are $n,n-2,\dots, -n$.

In what follows, $\VV=\VV_\lb$,  $\cP(\VV)$ is the set of weights of $\VV$, and  
\[\cP(\VV)_+=\{\nu\in\cP(\VV) \mid \nu(\tilde h)>0\} .\] 
We also write $\widetilde{\cP(\VV)_+}$ for the
{\sl multiset\/} of weights in $\cP(\VV)_+$,  where each $\nu$ appears with multiplicity $m_\lb^\nu$.
It is assumed below that $\lb\in\PT\cap Q$, so that $\cP(\VV)\subset Q$ and 
$m_\lb^0\ne 0$.

\begin{thm}   \label{thm:geom-tozhd}
Suppose that $\dim \VV^\te=\dim\VV^e(=:n)$. Then 
\begin{itemize}
\item[\sf (i)] \ $\cP(\VV)=\cP(\VV)_+\cup\{0\} \cup (-\cP(\VV)_+)$; moreover, each nonzero weight is a multiple of a root.
\item[\sf (ii)] \ 
$\displaystyle \prod_{\gamma\in\widetilde{\cP(\VV)_+}}
\frac{1-q^{\hot(\gamma)+1}}{1-q^{\hot(\gamma)}}=
\prod_{i=1}^{n}\frac{1-q^{m_i(\lb)+1}}{1-q}$.
\end{itemize}

\end{thm}
\begin{proof}
(i) \ Considering $\VV$ as $\es$-module, we obtain a decomposition
$\VV\vert_\es=\bigoplus_{j=1}^n \sfr_{l_j}$. 
Since $\dim\VV^e\ge\dim\VV^{\tilde h}\ge \dim\VV^\te$, the hypothesis implies that  each $\sfr_{l_j}$ has a zero-weight space (hence  each $l_j=2k_j$ is even) and $\VV^\te=\VV^{\tilde h}$. 
Consequently, if $\nu\in \cP(\VV)$ and $\nu\ne 0$, then $\nu(\tilde h)\ne 0$, which proves the partition
formula. Letting $\VV_\pm=\oplus_{\gamma\in\cP(\VV)_\pm}\VV^\gamma$, we see that $\VV^\te$
generates $\VV^\te\oplus\VV_+$ as $e$-module, whence $\cP(\VV)_+ \subset Q_+$.
Now, it is easily seen that if $\gamma\in Q_+$ is not proportional to a root, then there exists
$w\in W$ such that $w(\gamma)\not\in Q_+\cup (-Q_+)$.

(ii) \ By the above decomposition of $\VV\vert_\es$, the multiset of positive $\frac{1}{2}\tilde h$-eigenvalues in $\VV$, 
i.e., the multiset $\{\hot(\gamma) \mid \gamma\in \widetilde{\cP(\VV)_+}\}$ 
consists of $\{1,2,\dots,k_1,\,1,2,\dots,k_2,\,\dots ,\,1,2,\dots, k_n\}$. Therefore, most of the factors 
cancel out in the LHS  and we are left with the product
$\displaystyle \prod_{j=1}^n \frac{1-q^{k_i+1}}{1-q}$.

On the other hand, the theory of R.\,Brylinski~\cite{rkb89} shows that the generalised exponents of $\VV$ are determined by the $e$-filtration on $\VV^\te$ and are equal to the 
$\frac{1}{2}\tilde h$-eigenvalues in $V^{\z(e)}$, where $\z(e)$ is the centraliser of $e$ in $\g$.
Since  $\dim \VV^{\z(e)}=\dim\VV^\te$ for the simple $G$-modules having zero 
weight~\cite[Cor.\,2.7]{rkb89}, it follows that  $\VV^{\z(e)}=\VV^e$ in our situation, and
the eigenvalues in question are $k_1,k_2,\dots,k_n$. This yields the desired equality in part~(ii).
\end{proof}

\begin{rmk}   \label{rmk:formal-conseq}
A formal consequence of relation (ii) is that the partition $\bigl(m_1(\lb),\dots,m_n(\lb)\bigr)$ is dual 
to the partition formed by the numbers $\#\{\gamma\in\widetilde{\cP(\VV)_+} \mid \hot(\gamma)=i\}$.
For $\VV=\g$ and $\cP(\VV)_+=\Delta^+$, formula (ii) is sometimes 
called the Kostant-Macdonald identity, see \cite{ac89}; we also refer to~\cite{ac12} for a recent generalisation related to Schubert varieties.
\end{rmk}
\begin{rmk}   \label{rmk:graham}
If $\VV$ is a simple $G$-module with non-trivial zero-weight space, then 
\[
      \dim\VV^e\ge \dim\VV^{\z(e)}=\dim \VV^\te \le \dim \VV^{\tilde h}
\]
and $\dim\VV^e\ge \dim \VV^{\tilde h}$. Therefore the hypothesis of Theorem~\ref{thm:geom-tozhd}
implies that all these spaces have one and the same dimension. 
It is also known that, for {\sl any\/} simple $G$-module
$\VV$,  $\dim\VV^{\z(e)}$ equals the dimension of a largest weight space, which is achieved for 
either the unique dominant minuscule weight or zero, see \cite[Remark\,1.6]{gr92}.
\end{rmk}
Making use of the above coincidences and Theorem~\ref{thm:geom-tozhd}(i), one easily proves that the hypothesis of the theorem holds exactly for the following pairs $(\g,\lb)$ with simple $\g$:

\textbullet\quad $(\g,\theta)$ and $(\g,\theta_s)$, i.e., the adjoint and little adjoint representations of $\g$;

\textbullet\quad $(\GR{B}{r}, 2\vp_1)$,\ $(\GR{G}{2}, 2\vp_1)$,\ $(\GR{A}{1}, 2m\vp_1)$, $m\in\BN$. 
\\
The generalised exponents for the first two cases in the second line are:

$2,4,\dots,2r$ and $2,4,6$, respectively.

\end{document}